\documentclass[a4paper,12pt,reqno]{amsart}

\usepackage[utf8]{inputenc}
\usepackage[english]{babel}
\usepackage{amsmath,amssymb,amsfonts,amsthm}
\usepackage{mathrsfs}
\usepackage{graphicx}
\usepackage{booktabs}
\usepackage{float}
\usepackage[numbers,square,sort&compress]{natbib}
\usepackage[colorlinks,linkcolor=blue,anchorcolor=blue,citecolor=blue,urlcolor=black,pdfproducer={LaTex by Zhiyang Lyu}]{hyperref}
\usepackage[a4paper, margin = 3cm, bottom = 3cm]{geometry}

\usepackage{enumerate}
\usepackage{url}
\usepackage{pgfplots}
\usepgfplotslibrary{dateplot}

\usepackage{import}
\usepackage{pdfpages}
\usepackage{transparent}
\usepackage{xcolor}

\newtheorem{theorem}{Theorem}[section]
\newtheorem{corollary}[theorem]{Corollary}
\newtheorem{lemma}[theorem]{Lemma}
\newtheorem{proposition}[theorem]{Proposition}

\newtheorem*{xrem}{Remark}

\renewcommand{\Im}{\operatorname{Im}}

\def\ext{\mathrm{Ext}}
\def\teich{Teichm\"uller }

\title[The asymptoticity of extremal length in Teichm\"uller space] {The asymptoticity of extremal length in Teichm\"uller space}

\author[Zhiyang Lyu]{Zhiyang Lyu$^*$}
\address[Zhiyang Lyu] {School of Mathematical Sciences, University of Science and Technology of China, 96 Jinzhai Road, 230026, Hefei, Anhui, China}
\email{lyuzhiyang@ustc.edu.cn}
\thanks{$^*$Corresponding author: \href{mailto:lyuzhiyang@ustc.edu.cn}{lyuzhiyang@ustc.edu.cn}}

\author[Y. Qi]{Yi Qi}
\address [Yi Qi] {School of Mathematics and Systems Science, Beihang University, 100191, Beijing, China}
\email{yiqi@buaa.edu.cn}
\thanks{The second author was partially supported by the National Natural Science Foundation of China through Grant 12271017.}

\subjclass[2020]{30F60, 32G15, 57M50.}

\keywords{Teichm\"uller space, extremal length, Teichm\"uller ray.}

\makeatletter
\hypersetup{
    pdftitle={\@title},
    pdfauthor={Zhiyang Lyu, Yi Qi},
    pdfsubject={\@subjclass},
    pdfkeywords={\@keywords}
}
\makeatother

\begin{document}

\baselineskip=16pt
\parskip=2pt

\begin{abstract}
    We study the asymptotic behavior of extremal length along Teichm\"uller rays. Specifically, we determine the limit of extremal length along a Teichm\"uller ray and obtain an explicit expression for this limit, which complements a related formula established by Cormac Walsh. Building on this result and Kerckhoff's formula, we establish a formula for the limiting Teichm\"uller distance between two points moving along arbitrary pairs of Teichm\"uller rays. Furthermore, we derive a necessary and sufficient condition for two Teichm\"uller rays to be asymptotic. Finally, by shifting the initial points of the Teichm\"uller rays along their associated Teichm\"uller geodesics, we show that the minimum of the limiting Teichm\"uller distance coincides with the detour metric between the endpoints of the rays on the horofunction boundary.
\end{abstract}

\maketitle




\section{Introduction}\label{Introduction}

Extremal length is introduced by Ahlfors and Beurling \cite{Ahl2006} for families of curves in the plane. It is a conformal invariant and has the quasiconformal distortion property. Kerckhoff \cite{Ker1980} studied the extremal lengths of simple closed curves on Riemann surfaces, thereby proving that the Thurston compactification and \teich compactification are distinct. More significantly, he extended extremal length to the measured foliation space $\mathcal{MF}$ and showed a useful formula for the \teich distance, known as Kerckhoff's formula. Gardiner and Masur \cite{GM1991} further developed the extremal length geometry of \teich space. They showed that the infimum of the product of the extremal lengths for two transversely measured foliations is realized by a unique \teich geodesic. Furthermore, they gave a compactification of \teich space by the extremal length in a manner analogous to Thurston compactification. Walsh \cite{Wal2019} studied the asymptotic behavior of extremal length and established a formula for the limiting extremal length along a \teich ray. In addition, he proved that the Gardiner-Masur compactification coincides with the Thurston compactification. This result was also established independently by Liu and Su \cite{LS2014}.

Let $S$ be an oriented topological surface of genus $g$ with $m$ punctures such that $3g-3+m\geq 1$. The \teich space $\mathcal{T}(S)$ of $S$ is the space of marked Riemann surfaces up to \teich equivalence. It is endowed with a natural complete metric $d_{\mathcal{T}}(\cdot,\cdot)$, called the \teich metric. The \teich ray $\mathcal{R}_{q,X}(t)$ induced by a unit-norm quadratic differential $q$ on $X$ is a geodesic ray in $\mathcal{T}(S)$ with respect to the \teich metric. 

Let $V(q)$ and $H(q)$ denote the vertical and horizontal foliations of $q$ respectively. It is known that, by removing the critical graph $\Gamma_q$ of $V(q)$, the surface $X$ can be decomposed into finitely many cylinders swept out by closed leaves and minimal domains in which all leaves are dense. An indecomposable component of $V(q)$ is either a cylindrical component or a minimal component with a transverse ergodic measure. It can be shown that the number of projectively distinct ergodic measures on a minimal domain is finite.

Miyachi \cite{Miy2013} showed that a \teich ray admits a limit in the Gardiner-Masur compactification, and furthermore, that different \teich rays from $X$ have different limits. In particular, Miyachi \cite{Miy2008} provided an explicit formula for the limit in the case of Jenkins-Strebel rays, which is also proved by Kerckhoff \cite{Ker1980}. In \cite{Wal2019}, Walsh extended this formula to arbitrary \teich rays. For any measured foliation $F\in\mathcal{MF}$, Walsh showed that 
$$
\lim_{t\to\infty}e^{-2t}\ext_{\mathcal{R}_{q,X}(t)}(F)=\sum_{j=1}^n\frac{a_ji(G_j,F)^2}{i(G_j,H(q))},
$$
where $G_j$ is an indecomposable component of $V(q)$ and $a_j$ is a positive coefficient. Furthermore, Walsh proved that the limit of a \teich ray in the Gardiner-Masur boundary corresponds to a Busemann point in the horofunction boundary and vice versa. It was observed that non-Busemann points exist in the horofunction boundary, a result that was independently confirmed by Miyachi \cite{Miy2014}.

It is evident that the limit is degenerated if the measured foliation $F$ has no geometric intersection with $V(q)$. In order to ascertain the effect of the measured foliations that do not intersect with $V(q)$, it is necessary to consider the asymptotic behavior of $e^{2t}\ext_{\mathcal{R}_{q,X}(t)}(\cdot)$ for such measured foliations. The explicit limit expression for these measured foliations are derived as follows.

\begin{theorem}\label{limit of e^2text}
    Let $\mathcal{R}_{q,X}(t)$ be the Teichm\"uller ray from $X\in\mathcal{T}(S)$ determined by the unit-norm quadratic differential $q$. Suppose that the vertical foliation of $q$ can be expressed as $V(q)=\sum_{j=1}^na_jG_j$, where $a_j$ is a positive coefficient and $G_j$ is an indecomposable component of $V(q)$. Then for any $F\in\mathcal{MF}$, 
    $$
    \lim_{t\to\infty}e^{2t}\ext_{X_t}(F)=\sup_{F'\in\mathcal{MF}}\frac{i(F,F')^2}{\sum_{j=1}^n\frac{a_ji(G_j,F')^2}{i(G_j,H(q))}},
    $$
    where $X_t$ is a point on the \teich ray $\mathcal{R}_{q,X}(t)$.
\end{theorem}

The limit is infinite if the measured foliation $F$ intersects $V(q)$, since in this case the measured foliation $F'$ can be chosen to coincide with $V(q)$. Moreover, if the measured foliation $F$ is similar to $V(q)$, that is, if $F$ can be expressed as $F=\sum_{j=1}^nc_jG_j$, where $c_j> 0$ is a coefficient, the right side of the formula can be simplified.

\begin{corollary}\label{limit for similar foliation}
    Let $\mathcal{R}_{q,X}(t)$ be the Teichm\"uller ray determined by the unit-norm quadratic differential $q$, and let $V(q)=\sum_{j=1}^na_jG_j$ be the vertical foliation. If a measured foliation $F\in\mathcal{MF}$ can be written in the form $F=\sum_{j=1}^nc_jG_j$, where $c_j> 0$ is a coefficient, then 
    $$
    \lim_{t\to\infty}e^{2t}\ext_{X_t}(F)=\sum_{j=1}^n\frac{c_j^2i(G_j,H(q))}{a_j}.
    $$
\end{corollary}

The long-term behavior of \teich rays is a fundamental question in \teich geometric. In this paper, we focus on the asymptotic behavior of two \teich rays in \teich space. The divergence case for two \teich rays has been clarified by the efforts of Ivanov, Lenzhen and Masur. Ivanov \cite{Iva2001} showed that two \teich rays are divergent if their associated vertical foliations intersect. Furthermore, the \teich rays remain bounded if the vertical foliations are absolutely continuous (see \S\ref{measured foliations}). The remaining cases were resolved by Lenzhen and Masur \cite{LM2010}. They proved that two \teich rays are divergent if the vertical foliations are either topologically inequivalent or topologically equivalent but not absolutely continuous. Consequently, it can be concluded that two \teich rays are bounded if and only if the vertical foliations are absolutely continuous. 

For bounded \teich rays, the remaining questions concern the limiting \teich distances of pairs of \teich rays and the conditions under which they are asymptotic. In \cite{Mas1980}, Masur considered the case of uniquely ergodic foliations, proving that two bounded \teich rays with uniquely ergodic vertical foliations are asymptotic if their foliations contain no simple closed curves formed by saddle connections. Amano \cite{Ama2014-TheAsym} derived a formula for the limiting \teich distance of Jenkins-Strebel rays, showing that the limiting distance of bounded Jenkins-Strebel rays depends on both the moduli of cylindrical components of initial surfaces and the \teich distance between their endpoints in the augmented \teich spaces. In \cite{HLMQ}, we extended this formula to \teich rays along which the surfaces can be decomposed into cylinders and minimal domains whose transverse measures are uniquely ergodic, with the decomposition determined by the leaves of the vertical foliations. 

In \cite{HLMQ}, we adopted an approach based on constructing quasiconformal mappings between surfaces along two bounded \teich rays, combining techniques due to Amano and Masur. However, this method does not readily extend to the case of \teich rays whose associated vertical foliations contain non-uniquely ergodic minimal components. In this paper, we instead study the limiting \teich distances via Kerckhoff's formula. This approach allows us to treat the remaining case of non-uniquely ergodic foliations and avoids the technically demanding construction of quasiconformal mappings.

Since the vertical foliations $V(q)$ and $V(q')$ are absolutely continuous for the two bounded \teich rays $\mathcal{R}_{q,X}(t)$ and $\mathcal{R}_{q',Y}(t)$, each measured foliation $F\in\mathcal{MF}$ either intersects both $V(q)$ and $V(q')$ or intersects neither. Let $\mathcal{MF}_1(V(q))$ denote the collection of measured foliations that intersect $V(q)$, and let $\mathcal{MF}_0(V(q))$ denote the collection of measured foliations that have zero intersection with $V(q)$. We investigate the ratios of extremal lengths of measured foliations along the two bounded \teich rays and obtain two equalities for the limits of the supremum of these ratios, taken respectively over $\mathcal{MF}_1(V(q))$ and $\mathcal{MF}_0(V(q))$ (see Theorem \ref{equality of Ext in MF_1} and Theorem \ref{equality of Ext in MF_0}).
$$
\lim_{t\to\infty}\sup_{F\in\mathcal{MF}_1(V(q))}\frac{\mathrm{Ext}_{Y_t}(F)}{\mathrm{Ext}_{X_t}(F)}=\max_{1\leq j\leq n}\frac{b_ji(G_j,H(q))}{a_ji(G_j,H(q'))}.
$$
$$
\lim_{t\to\infty}\sup_{F\in\mathcal{MF}_0(V(q))}\frac{\mathrm{Ext}_{Y_t}(F)}{\mathrm{Ext}_{X_t}(F)}=\max_{1\leq j\leq n}\frac{a_ji(G_j,H(q'))}{b_ji(G_j,H(q))}.
$$
Combining these results with Kerckhoff's formula, we obtain the limiting distance for arbitrary pairs of \teich rays through the asymptotic behavior of extremal lengths, thereby establishing an explicit expression for the limit for any such pair.

\begin{theorem}\label{limiting distance}
    Let $\mathcal{R}_{q,X}(t)$ and $\mathcal{R}_{q',Y}(t)$ be two \teich rays induced by the unit-norm quadratic differentials $q$ and $q'$ respectively.
    \begin{itemize}
        \item [(i)] If the vertical foliations $V(q)$ and $V(q')$ are absolutely continuous, i.e. $V(q)=\sum_{j=1}^na_jG_j$ and $V(q')=\sum_{j=1}^nb_jG_j$, where $a_j, b_j>0$ are coefficients, then
        $$
        \lim_{t\to\infty}d_{\mathcal{T}}(X_t,Y_t)=\frac{1}{2}\log\max_{1\leq j\leq n}\left\{\frac{a_ji(G_j,H(q'))}{b_ji(G_j,H(q))},\frac{b_ji(G_j,H(q))}{a_ji(G_j,H(q'))}\right\}.
        $$
        \item [(ii)] Otherwise, 
        $$\lim_{t\to\infty}d_{\mathcal{T}}(X_t,Y_t)=+\infty.$$
    \end{itemize} 
\end{theorem}

\begin{xrem}
    In \cite{HLMQ}, We defined the limiting surfaces for \teich rays and the \teich distance between these limiting surfaces. A comparison between the main result of \cite{HLMQ} and Theorem \ref{limiting distance} shows that, when considering the limit distance for two bounded \teich rays, it is unnecessary to analyze the distance between their limiting surfaces. Indeed, since the intersection number of a simple closed curve with $V(q)$ is preserved along the \teich ray $\mathcal{R}_{q,X}(t)$, each simple closed curve in $\mathcal{MF}_0(V(q))$ corresponds naturally to a simple closed curve on the limiting surface of $\mathcal{R}_{q,X}(t)$, which may be homotopic to a puncture. Then, by combining Kerckhoff's formula with Theorem \ref{equality of Ext in MF_0}, Theorem \ref{limiting distance} implies that the distance between limiting surfaces of two bounded \teich rays $\mathcal{R}_{q,X}(t)$ and $\mathcal{R}_{q',Y}(t)$ is bounded by $\frac{1}{2}\log\max_{1\leq j\leq n}\left\{\frac{a_ji(G_j,H(q'))}{b_ji(G_j,H(q))},\frac{b_ji(G_j,H(q))}{a_ji(G_j,H(q'))}\right\}$.
\end{xrem}

The notion of modulus for a cylinder extends naturally to any indecomposable component of $V(q)=\sum_{j=1}^na_jG_j$. We define the modulus of the indecomposable component $G_j$ of $V(q)$ by
$$
m_j=\frac{a_j}{i(G_j,H(q))}.
$$
Given two absolutely continuous vertical foliations $V(q)=\sum_{j=1}^na_jG_j$ and $V(q')=\sum_{j=1}^nb_jG_j$, we say that $q$ and $q'$ (or $V(q)$ and $V(q')$) are modularly equivalent if 
$$
\frac{a_j}{i(G_j,H(q))}=C\frac{b_j}{i(G_j,H(q'))},
$$
for all $j$, where $C$ is a positive constant.

\begin{corollary}\label{asymptotic}
    The \teich rays $\mathcal{R}_{q,X}(t)$ and $\mathcal{R}_{q',Y}(t)$ are asymptotic if and only if $q$ and $q'$ are modularly equivalent.
\end{corollary}

The limiting \teich distance between any pair of \teich rays only depends on the ratios of the moduli of the indecomposable components at the initial points. Therefore, By appropriately shifting these initial points along their corresponding geodesics, the minimum of the limiting \teich distance can be achieved, which is associated with the detour metric on the horofunction boundary of $\mathcal{T}(S)$.

\begin{proposition}\label{minimum of the limiting distance}
    Let $\mathcal{R}_{q,X}(t)$ and $\mathcal{R}_{q',Y}(t)$ be two \teich rays induced by the unit-norm quadratic differentials $q$ and $q'$ respectively. If the vertical foliations $V(q)$ and $V(q')$ are absolutely continuous, then the minimum of the limiting \teich distance is $\frac{1}{2}\delta(\mathcal{B}_q,\mathcal{B}_{q'})$, where $\delta$ is the detour metric and $\mathcal{B}_q,\mathcal{B}_{q'}$ are the Busemann points on the horofunction boundary of $\mathcal{T}(S)$ associated to $\mathcal{R}_{q,X}(t)$ and $\mathcal{R}_{q',Y}(t)$ respectively.
\end{proposition}

The following corollary, which has been proved by Walsh in \cite{Wal2019}, is a direct combination of Corollary \ref{asymptotic} and Proposition \ref{minimum of the limiting distance}.

\begin{corollary}\label{identical Busemann points}
    The Busemann points $\mathcal{B}_q$ and $\mathcal{B}_{q'}$ are identical if and only if $q$ and $q'$ are modularly equivalent.
\end{corollary}

This paper is organized as follows. In Section \ref{section2}, we recall some background on \teich spaces, measured foliations, quadratic differentials, extremal lengths and \teich rays. In Section \ref{section3}, we investigate the asymptotic behavior of extremal length along \teich rays based on the results of Walsh \cite{Wal2019}. Then we prove Theorem \ref{limit of e^2text} and Corollary \ref{limit for similar foliation}. In Section \ref{section4}, the ratios of extremal lengths along pairs of \teich rays is explored. Furthermore, by Kerckhoff's formula, we prove Theorem \ref{limiting distance} and Corollary \ref{asymptotic}. In Section \ref{section5}, we review some background and results on the horofunction compactification of \teich space and detour metric, and we prove Proposition \ref{minimum of the limiting distance} and Corollary \ref{identical Busemann points}.


\section{Preliminaries}\label{section2}

\subsection{Teichm\"uller spaces}

Let $X$ be a Riemann surface and $f$ be an orientation-preserving homeomorphism from $S$ onto $X$. The pair $(X,f)$ is called a marked Riemann surface, where $f$ serves as the marking of $X$. Two marked Riemann surfaces $(X_1,f_1)$ and $(X_2,f_2)$ are said to be \teich equivalent if there exists a conformal homeomorphism $\phi:X_1\to X_2$ such that $f_2^{-1}\circ\phi\circ f_1$ is homotopic to the identity mapping on $S$. The \teich equivalent classes containing $(X,f)$ is denoted by $[X,f]$, and the \teich space $\mathcal{T}(S)$ of $S$ is defined as the collection of all such equivalence classes. For notational simplicity, we usually use the Riemann surface $X$ to represent the class $[X,f]$. 

The \teich distance between two Riemann surfaces $X_1, X_2\in\mathcal{T}(S)$ is given by:
$$
d_{\mathcal{T}}(X_1,X_2)=\frac{1}{2}\inf_{h}\{\log K(h)\},
$$
where the infimum is over all quasiconformal mapping $h:X_1\to X_2$ such that $f_2^{-1}\circ h\circ f_1$ is homotopic to the identity mapping on $S$, and $K(h)$ denotes the maximal quasiconformal dilatation of $h$. The \teich space $\mathcal{T}(S)$ with the \teich metric $d_\mathcal{T}$ is a complete geodesic metric space.


\subsection{Measured foliations}\label{measured foliations}

Let $P$ be a finite set composed of finite points and punctures on $S$. A measured foliation on $S$ consists of a finite set $P$ and an open covering $\{U_j\}$ of $S\setminus P$ with a real-valued $C^1$-function $v_j$ on each $U_j$ such that $|dv_j|=|dv_i|$ on $U_j\cap U_i$. Each $p\in P$ is called a singularity of the measured foliation, and there is a neighborhood $V_p$ of $p$ with complex coordinate $z$ such that $|dv_j|=|\Im (z^{k/2}dz)|$ on $U_j\cap V_p$ for some $k\in\mathbb{Z}$. If $p\in P$ is a puncture of $S$, $k\geq -1$. Otherwise, $k\geq 0$. We denote by $(\mathcal{F},\mu)$ the measured foliation on $S$, where $\mu$ is a transverse invariant measure which is given by $|dv_j|$ in $U_j$, and $\mathcal{F}$ denotes the underlying foliation structure without considering the measure. The leaves of $\mathcal{F}$ correspond to curves on $S$ along which the functions $v_j$ remain constant.

Let $\mathcal{S}$ be the set of homotopy classes of essential non-peripheral unoriented simple closed curves on $S$. The intersection number of a measured foliation $(\mathcal{F},\mu)$ and a $\alpha\in\mathcal{S}$ is defined by
$$
i((\mathcal{F},\mu),\alpha)=\inf_{\alpha'\in\alpha}\int_{\alpha'}d\mu,
$$
where the infimum is taken over all simple closed curves in $\alpha$. Two measured foliations $(\mathcal{F}_1,\mu_1)$ and $(\mathcal{F}_2,\mu_2)$ are equivalent if 
$$
i((\mathcal{F}_1,\mu_1),\alpha)=i((\mathcal{F}_2,\mu_2),\alpha)
$$
holds for any $\alpha\in\mathcal{S}$. Equivalently, they are Whitehead equivalence, meaning that one can be transformed into the other by a finite sequence of homeomorphisms of the surface isotopic to the identity, which preserve the transverse measures, together with Whitehead moves, where a Whitehead move is a deformation that collapses a leaf joining two singularities to a point, which may be a puncture, as well as the inverse move of such a collapse.

We denote by $F$ the equivalent class containing $(\mathcal{F},\mu)$. Therefore, the function $i(F,\cdot)$ is an element of $\mathbb{R}^\mathcal{S}_{\geq 0}$, where $\mathbb{R}^\mathcal{S}_{\geq 0}$ is the space of non-negative functions on $\mathcal{S}$ topologized by pointwise convergence. We denote by $\mathcal{MF}$ the space of equivalent classes of measured foliations on $S$ equipped with the weak topology induced by the intersection functions in $\mathbb{R}^\mathcal{S}_{\geq 0}$. The intersection function can extend continuously to $\mathcal{MF}\times\mathcal{MF}$ (cf. \cite{Bon1988} and \cite{Ree1981}).

A saddle connection of $F$ is a leaf of $F$ joining two singularities, which may be identical. The union of all saddle connections of $F$ is denoted by $\Gamma_F$. It is known that each connected component of $S\setminus\Gamma_F$ is either a cylinder foliated by closed leaves or a minimal domain in which each leaf is dense. For a minimal domain $\Omega$ of $F$, there exists finite number of ergodic transverse measures $\{\mu_i\}_{i=1}^k$ on $\Omega$, where $k$ is determined by the topology of $\Omega$, such that any transverse invariant measure $\mu$ on $\Omega$ can be expressed as $\mu=\sum_{i=1}^ka_i\mu_i$ for $a_i\geq 0$. The measured foliation $F$ is uniquely ergodic on a minimal domain $\Omega$ if the transverse measure on $\Omega$ is unique up to scalar multiplication. We say that a measured foliation $F$ is uniquely ergodic if $S\setminus\Gamma_F$ contains only one minimal domain with a uniquely ergodic transverse measure. Then, the measured foliation $F$ can be represented as
$$
F=\sum_{j=1}^na_jG_j,
$$
where $G_j$ is either a cylindrical component or a minimal domain endowed with an ergodic transverse measure. The coefficients satisfy $a_j>0$ when $G_j$ is a cylindrical component and $a_j\geq 0$ when $G_j$ is a minimal component.

Two measured foliations $F_1, F_2\in\mathcal{MF}$ are topologically equivalent if there exists a homeomorphism $f:S\setminus\Gamma_{F_1}\to S\setminus\Gamma_{F_2}$ that is isotopic to the identity and maps each leaf of $F_1$ to a corresponding leaf of $F_2$. In this case, $F_1$ and $F_2$ can be expressed as
$$
F_1=\sum_{j=1}^na_jG_j,\quad F_2=\sum_{j=1}^nb_jG_j,
$$
where $G_j$ is either a cylindrical component or a minimal component. Furthermore, we say that $F_1$ and $F_2$ are absolutely continuous if they are topologically equivalent and their corresponding coefficients with positive values coincide. Then, for absolutely continuous foliations, we retain only the components with positive coefficients such that $a_j,b_j>0$ for all $j=1,\cdots,n$.

Since there is a natural bijection between measured foliations on any two surfaces in $\mathcal{T}(S)$, achieved through pullback of measured foliations to the corresponding measured foliations on $S$, we can extend the notions of topologically equivalent and absolutely continuous to pairs of measured foliations on distinct Riemann surfaces. For simplicity, we shall denote by $F\in\mathcal{MF}$ the corresponding measured foliation on any surface $X\in\mathcal{T}(S)$.


\subsection{Quadratic differentials}

A holomorphic quadratic differential $q$ on a Riemann surface $X$ is a tensor of the form $q(z)dz^2$, where $q$ is a holomorphic function that is allowed to have simple poles at the punctures of $X$. Each $q$ induces a singular flat metric $|q|=|q(z)||dz|^2$ on $X$. The zeros and poles of $q$ refer to the critical points of $q$, which form the singularities of the flat metric on $X$. The space of all holomorphic quadratic differentials on $X$ is denoted by $Q(X)$ and is a Banach space with the norm $\|q\|=\int_X|q(z)|dxdy$. 

Let $q$ be a holomorphic quadratic differential on $X$ and $I\subset\mathbb{R}$ be an interval. A trajectory of $q$ is a smooth path $\gamma:I\to X$ on $X$ such that the $\arg{q(\gamma(t))\gamma'(t)^2}$ is constant on $I$. A trajectory of $q$ is maximal if it is not a proper sub-trajectory of any trajectory on $X$. There exists a canonical coordinate chart $z=x+iy$ for each point in $X$ except for critical points of $q$, such that $q=dz^2$ in the chart. A vertical trajectory of $q$ is a maximal trajectory on $X$ such that it is a vertical line in each canonical coordinate chart. The collection of all critical trajectories of $q$ forms a measured foliation $V(q)$ on $X$, called the vertical foliation of $q$, whose transverse measure is given by $|dx|$. Similarly, a horizontal trajectory of $q$ is a maximal trajectory on $X$ which is a horizontal line in a canonical coordinate chart. The horizontal foliation $H(q)$ of $q$ is formed by all horizontal trajectories of $q$, equipped with the transverse measure induced by $|dy|$.

A critical trajectory of $V(q)$ is a vertical trajectory with an endpoint at a critical point of $q$. The critical graph $\Gamma_q$ of $V(q)$ is the union of all critical trajectories of $V(q)$ connecting critical points of $q$. A holomorphic quadratic differential $q$ is called a Jenkins-strebel differential if all the leaves of $V(q)$ are closed except for the critical graph $\Gamma_q$. Hubbard and Masur \cite{HM1979} showed that for each measured foliation $F\in\mathcal{MF}$, there exists a holomorphic quadratic differential $q$ on $X$ whose vertical foliation $V(q)$ realizes the measured foliation $F$.


\subsection{Extremal length}

We recall the extremal length of a simple closed curve $\alpha\in\mathcal{S}$ on $[X,f]\in\mathcal{T}(S)$. Let $\rho=\rho(z)|dz|$ be a Borel measurable conformal metric on $X$. Under this metric $\rho$, the length of $\alpha\in\mathcal{S}$ on $X$ is 
$$
\ell_\rho(\alpha)=\inf_{\gamma'\in f(\alpha)}\int_{\gamma'}\rho(z)|dz|,
$$
where $\gamma'$ ranges over all simple closed curves in the homotopy class $f(\alpha)$ corresponding to $\alpha\in\mathcal{S}$. The area of $X$ is 
$$
\mathrm{Area}_\rho(X)=\iint_X\rho(z)^2dxdy.
$$
The extremal length $\mathrm{Ext}_X(\alpha)$ of $\alpha$ on $X$ is defined by 
$$
\mathrm{Ext}_X(\alpha)=\sup_\rho\frac{\ell_\rho(\alpha)^2}{\mathrm{Area}_\rho(X)},
$$
where $\rho$ takes over all Borel measurable conformal metric on $X$ with $\mathrm{Area}_\rho(X)<\infty$.

Kerckhoff \cite{Ker1980} showed that there is a unique continuous extension of the extremal length function from $\mathcal{S}$ to $\mathcal{MF}$ satisfying
$$
\mathrm{Ext}_X(tF)=t^2\mathrm{Ext}_X(F), \quad t\in\mathbb{R}_{>0},
$$
for any $F\in\mathcal{MF}$ and $X\in\mathcal{T}(S)$. Moreover, Kerckhoff established a useful formula of the \teich distance by extremal length.

\begin{lemma}[Kerckhoff's formula]\label{Kerckhoff's formula}
    The \teich distance between $X, Y\in\mathcal{T}(S)$ is 
    $$
    d_\mathcal{T}(X,Y)=\frac{1}{2}\log\sup_{\alpha\in\mathcal{S}}\frac{\mathrm{Ext}_Y(\alpha)}{\mathrm{Ext}_X(\alpha)}.
    $$
\end{lemma}

\begin{xrem}
    Since the unique continuous extension of the extremal length function to $\mathcal{MF}$ and the density of the set $\{t\alpha\mid t\in\mathbb{R}_{>0}, \alpha\in\mathcal{S}\}$ in $\mathcal{MF}$, the \teich distance between $X, Y\in\mathcal{T}(S)$ is also given by
    $$
    d_\mathcal{T}(X,Y)=\frac{1}{2}\log\sup_{F\in\mathcal{MF}}\frac{\mathrm{Ext}_Y(F)}{\mathrm{Ext}_X(F)}.
    $$
    Let $\mathcal{PMF}$ denote the projective space of measured foliations and $[F]$ be the equivalent class in $\mathcal{PMF}$ containing $F\in\mathcal{MF}$. Then, we also have
    $$
    d_\mathcal{T}(X,Y)=\frac{1}{2}\log\sup_{[F]\in\mathcal{PMF}}\frac{\mathrm{Ext}_Y(F)}{\mathrm{Ext}_X(F)}.
    $$
\end{xrem}

Extremal length and intersection number are continuous on $\mathcal{MF}$, and a useful inequality between them can be derived from the work by Gardiner and Masur in \cite{GM1991}, which is also given by Minsky \cite{Min1993} and Ivanov \cite{Iva2001}.

\begin{lemma}\label{inequality}
    Let $F$ and $G$ be two measured foliations in $\mathcal{MF}$ and $X\in\mathcal{T}(S)$. Then
    $$
    \mathrm{Ext}_X(F)\mathrm{Ext}_X(G)\geq i(F,G)^2.
    $$
    Moreover, the inequality will be an equality if and only if $F$ and $G$ are realized by the horizontal and vertical foliations of a holomorphic quadratic differential on $X$.
\end{lemma}

For any $F\in\mathcal{MF}$ and $X\in\mathcal{T}(S)$, there is a holomorphic quadratic differential $q$ on $X$ with $F$ as its vertical foliation by Hubbard and Masur \cite{HM1979}. Then, there exists a measure foliation $G$ which is the horizontal foliation of $q$, such that the equality in Lemma \ref{inequality} holds for $G$.


\subsection{Teichm\"uller rays}\label{Teichmuller rays}

Consider the quasiconformal mappings between two Riemann surfaces $X,Y\in\mathcal{T}(S)$. Teichm\"uller's theorem states that there exists a unique quasiconformal mapping that realizes the \teich distance $d_\mathcal{T}(X,Y)$, known as the \teich mapping. For the \teich mapping $f:X\to Y$, its Beltrami coefficient $\mu_f$ has the form $\mu_f=\frac{K(f)-1}{K(f)+1}\frac{\bar{q}}{|q|}$, where $q$ is a unit-norm holomorphic quadratic differential on $X$.

Let $f_{q,t}$ be the \teich mapping from $X$ to $X_t$ determined by the unit-norm holomorphic quadratic differential $q$ on $X$. There is a unit-norm holomorphic quadratic differential $q_t$ on $X_t$, called the terminal differential. Then, in the canonical coordinate $z=x+iy$ of $q$ and the canonical coordinate of $q_t$, the \teich mapping $f_{q,t}$ is given by 
$$
z\mapsto e^tx+e^{-t}y,
$$
where $t=\frac{1}{2}\log K(f_{q,t})$ is the arc length parameter with respect to \teich distance. 

The \teich geodesic $\mathcal{G}_{q,X}(t)$ through $X$ induced by a unit-norm holomorphic quadratic differential $q$ can be defined by 
\begin{equation*}
    \begin{array}{cccl}
        \mathcal{G}_{q,X}: & \mathbb{R} & \to & \mathcal{T}(S) \\
         & t & \mapsto & X_t=f_{q,t}(X), \\
    \end{array}
\end{equation*} 
where $f_{q,t}:X\to X_t$ is the Teichm\"uller mapping for the holomorphic quadratic differential $q$ on $X$. The \teich ray emanating from $X$ and induced by $q$ is defined by $\mathcal{R}_{q,X}(t):=\left.\mathcal{G}_{q,X}\right|_{\mathbb{R}_{\geq 0}}(t)$.

Let $\mathcal{R}_{q,X}(t)$ and $\mathcal{R}_{q^\prime,Y}(t)$ be two \teich rays. 
\begin{itemize}
    \item[(i)] The two \teich rays are divergent if 
    $$
    d_\mathcal{T}(X_t,Y_t)\to +\infty \text{ as } t\to\infty.
    $$
    \item[(ii)] The two \teich rays are bounded if there exists a constant $M>0$ such that $d_\mathcal{T}(X_t,Y_t)<M$ for any $t\geq 0$.
    \item[(iii)] The two \teich rays are asymptotic if 
    $$
    \lim_{t\to\infty}\inf_{Y'\in\mathcal{R}_{q',Y}(t)}d_{\mathcal{T}}(X_t,Y')=0.
    $$
    Then, there is a $\sigma\in\mathbb{R}$ such that $d_{\mathcal{T}}(X_t,Y_{t+\sigma})\to 0$ as $t\to\infty$.
\end{itemize}


\section{The asymptotic behavior of extremal length}\label{section3}

We investigate the asymptotic behavior of $e^{2t}\mathrm{Ext}_{X_t}(F)$ for any $F\in\mathcal{MF}$ along a \teich ray $\mathcal{R}_{q,X}(t)$. This study serves as a complement to the results concerning the asymptotic behavior of $e^{-2t}\mathrm{Ext}_{X_t}(F)$ in \cite{Wal2019}.

We adopt the convention that the supremum of a function is evaluated exclusively over the domain where the function is well-defined. Specifically, We consider $c/0=\infty$ to be well-defined for any $c>0$, while $0/0$ and $\infty/\infty$ remain undefined. 

\begin{lemma}[{\cite[Lemma~3]{Wal2019}}]\label{Walsh's inequality}
    Let $\mathcal{R}_{q,X}(t)$ be a \teich ray induced by a unit-norm holomorphic quadratic differential $q$ on $X$. The vertical foliation of $q$ is $V(q)=\sum_{j=1}^na_jG_j$, where $a_j\geq 0$. Then,
    $$
    e^{-2t}\mathrm{Ext}_{X_t}(F)\geq\sum_{j=1}^n\frac{a_ji(G_j,F)^2}{i(G_j,H(q))},
    $$
    for any $F\in\mathcal{MF}$ and $t\geq 0$.
\end{lemma}

\begin{theorem}[{\cite[Theorem~1]{Wal2019}}]\label{Walsh's equality}
    Let $\mathcal{R}_{q,X}(t)$ be a \teich ray induced by a unit-norm holomorphic quadratic differential $q$ on $X$, and $V(q)=\sum_{j=1}^na_jG_j$, where $a_j\geq 0$. Then, 
    $$
    \lim_{t\to\infty}e^{-2t}\mathrm{Ext}_{X_t}(F)=\sum_{j=1}^n\frac{a_ji(G_j,F)^2}{i(G_j,H(q))},
    $$
    for any $F\in\mathcal{MF}$.
\end{theorem}

Based on the results given by Walsh in \cite{Wal2019} and Lemma \ref{inequality}, we show the limit of $e^{2t}\mathrm{Ext}_{X_t}(F)$ for any $F\in\mathcal{MF}$ along a \teich ray.

\begin{proof}[\textbf{Proof of Theorem} \ref{limit of e^2text}]
    By Lemma \ref{inequality}, for any measured foliations $F,F'\in\mathcal{MF}$, we have
    $$
    \mathrm{Ext}_{X_t}(F)\mathrm{Ext}_{X_t}(F')\geq i(F,F')^2.
    $$
    Then, by Theorem \ref{Walsh's equality}, there is
    $$
    \liminf_{t\to\infty}e^{2t}\mathrm{Ext}_{X_t}(F)\geq\liminf_{t\to\infty}\frac{i(F,F')^2}{e^{-2t}\mathrm{Ext}_{X_t}(F')}=\frac{i(F,F')^2}{\sum_{j=1}^n\frac{a_ji(G_j,F')^2}{i(G_j,H(q))}}.
    $$
    Therefore,
    \begin{equation}\label{limtinf of e^2text}
        \liminf_{t\to\infty}e^{2t}\mathrm{Ext}_{X_t}(F)\geq\sup_{F'\in\mathcal{MF}}\frac{i(F,F')^2}{\sum_{j=1}^n\frac{a_ji(G_j,F')^2}{i(G_j,H(q))}}.
    \end{equation}

    For any $F\in\mathcal{MF}$ and $X_t$ on the \teich ray $\mathcal{R}_{q,X}(t)$, there exists a $H_t\in\mathcal{MF}$ by Lemma \ref{inequality}, such that 
    $$
    \mathrm{Ext}_{X_t}(F)\mathrm{Ext}_{X_t}(H_t)=i(F,H_t)^2.
    $$
    By Lemma \ref{Walsh's inequality}, we have
    $$
    e^{2t}\mathrm{Ext}_{X_t}(F)=\frac{i(F,H_t)^2}{e^{-2t}\mathrm{Ext}_{X_t}(H_t)}\leq\frac{i(F,H_t)^2}{\sum_{j=1}^n\frac{a_ji(G_j,H_t)^2}{i(G_j,H(q))}}\leq\sup_{F'\in\mathcal{MF}}\frac{i(F,F')^2}{\sum_{j=1}^n\frac{a_ji(G_j,F')^2}{i(G_j,H(q))}}.
    $$
    Then, 
    \begin{equation}\label{limtsup of e^2text}
        \limsup_{t\to\infty}e^{2t}\mathrm{Ext}_{X_t}(F)\leq\sup_{F'\in\mathcal{MF}}\frac{i(F,F')^2}{\sum_{j=1}^n\frac{a_ji(G_j,F')^2}{i(G_j,H(q))}}.
    \end{equation}

    By combining \eqref{limtinf of e^2text} and \eqref{limtsup of e^2text}, the existence of the limit of $e^{2t}\mathrm{Ext}_{X_t}(F)$ along the \teich ray $\mathcal{R}_{q,X}(t)$ is obtained, and we have 
    $$
    \lim_{t\to\infty}e^{2t}\mathrm{Ext}_{X_t}(F)=\sup_{F'\in\mathcal{MF}}\frac{i(F,F')^2}{\sum_{j=1}^n\frac{a_ji(G_j,F')^2}{i(G_j,H(q))}}
    $$
    for any $F\in\mathcal{MF}$.
\end{proof}

\begin{proof}[\textbf{Proof of Corollary} \ref{limit for similar foliation}]
    If the measured foliation $F$ can be expressed as $F=\sum_{j=1}^nc_jG_j$, where $c_j> 0$, then, by the Cauchy-Schwarz inequality, there is 
    $$
    \frac{i(F,F')^2}{\sum_{j=1}^n\frac{a_ji(G_j,F')^2}{i(G_j,H(q))}}=\frac{\left(\sum_{j=1}^nc_ji(G_j,F')\right)^2}{\sum_{j=1}^n\frac{a_ji(G_j,F')^2}{i(G_j,H(q))}}\leq\sum_{j=1}^n\frac{c_j^2i(G_j,H(q))}{a_j}
    $$
    for any $F'\in\mathcal{MF}$.

    The equality holds if and only if there is a $F'\in\mathcal{MF}$ such that 
    $$
    a_ji(G_j,F')=c_ji(G_j,H(q))
    $$
    for each indecomposable component $G_j$. This $F'$ satisfying the condition can be obtained by appropriately adjusting the horizontal foliation $H(q)$ in the intersection set with $G_j$. Let $\Omega$ be a minimal domain of $V(q)$, and fix a horizontal arc $\tau$ in $\Omega$. Any leaf starting on $\tau$ either terminates at a singularity of $V(q)$ or returns to $\tau$. Let $\mu_j$, $j\in J_\Omega\subseteq\{1,\cdots,n\}$, denote the ergodic transverse measure of $G_j$ supported on $\Omega$. Since the transverse measures $\{\mu_j\}_{j\in J_\Omega}$ are mutually singular, the horizontal arc $\tau$ admits a decomposition 
    $$
    \tau=\bigcup_{j\in J_\Omega}\tau_j,
    $$
    where $\{\tau_j\}_{j\in J_\Omega}$ are pairwise disjoint Borel subsets such that $\mu_j(\tau_k)=\mu_j(\tau)$ when $j=k$ and $\mu_j(\tau_k)=0$ otherwise. 
    
    Consider the lengths of leaves in $\Omega$ with respect to the flat metric induced by $q$. For each $j\in J_\Omega$, we rescale the lengths of leaves crossing $\tau_j$ by the factor $c_j/a_j$. Since each leaf in $\Omega$ has infinite length, this rescaling preserves the underlying foliation structure of $V(q)$ on $\Omega$. Similarly, for each cylindrical component $G_j$ of $V(q)$, we rescale the lengths of the closed leaves in $G_j$ by the factor $c_j/a_j$. Then, the lengths of saddle connections on the boundaries of cylinders and minimal domains are also rescaled by the corresponding factors. 

    Since a saddle connection of $V(q)$ may have different lengths on the boundaries of two cylinders or minimal domains, it is necessary to modify the leaves joining singularities such that the adjacent cylinders or minimal domains match along the leaves. This can be achieved by introducing additional singularities and a saddle connection. Suppose that a saddle connection has length $\ell_i$ on one boundary and $\ell_j$ on the other boundary, with $\ell_i>\ell_j$. Consider the saddle connection of length $\ell_i$. We pick three points on this saddle connection such that it is divided into four arcs, with the two middle subarcs having equal length $\frac{\ell_i-\ell_j}{2}$. By identifying these two middle subarcs, we obtain a saddle connection joining a one-prong singularity to a three-prong singularity (see Figure~\ref{fig:saddle_connection}). Since the remaining length of the saddle connection is $\ell_j$, the two corresponding components of $V(q)$ can be glued together along their boundaries. 

    \begin{figure}[!htpb]
        \centering
    \fontsize{8pt}{10pt}\selectfont
    \def\svgwidth{1\columnwidth}
\begingroup%
  \makeatletter%
  \providecommand\color[2][]{%
    \errmessage{(Inkscape) Color is used for the text in Inkscape, but the package 'color.sty' is not loaded}%
    \renewcommand\color[2][]{}%
  }%
  \providecommand\transparent[1]{%
    \errmessage{(Inkscape) Transparency is used (non-zero) for the text in Inkscape, but the package 'transparent.sty' is not loaded}%
    \renewcommand\transparent[1]{}%
  }%
  \providecommand\rotatebox[2]{#2}%
  \newcommand*\fsize{\dimexpr\f@size pt\relax}%
  \newcommand*\lineheight[1]{\fontsize{\fsize}{#1\fsize}\selectfont}%
  \ifx\svgwidth\undefined%
    \setlength{\unitlength}{1232.26410633bp}%
    \ifx\svgscale\undefined%
      \relax%
    \else%
      \setlength{\unitlength}{\unitlength * \real{\svgscale}}%
    \fi%
  \else%
    \setlength{\unitlength}{\svgwidth}%
  \fi%
  \global\let\svgwidth\undefined%
  \global\let\svgscale\undefined%
  \makeatother%
  \begin{picture}(1,0.42782013)%
    \lineheight{1}%
    \setlength\tabcolsep{0pt}%
    \put(0,0){\includegraphics[width=\unitlength,page=1]{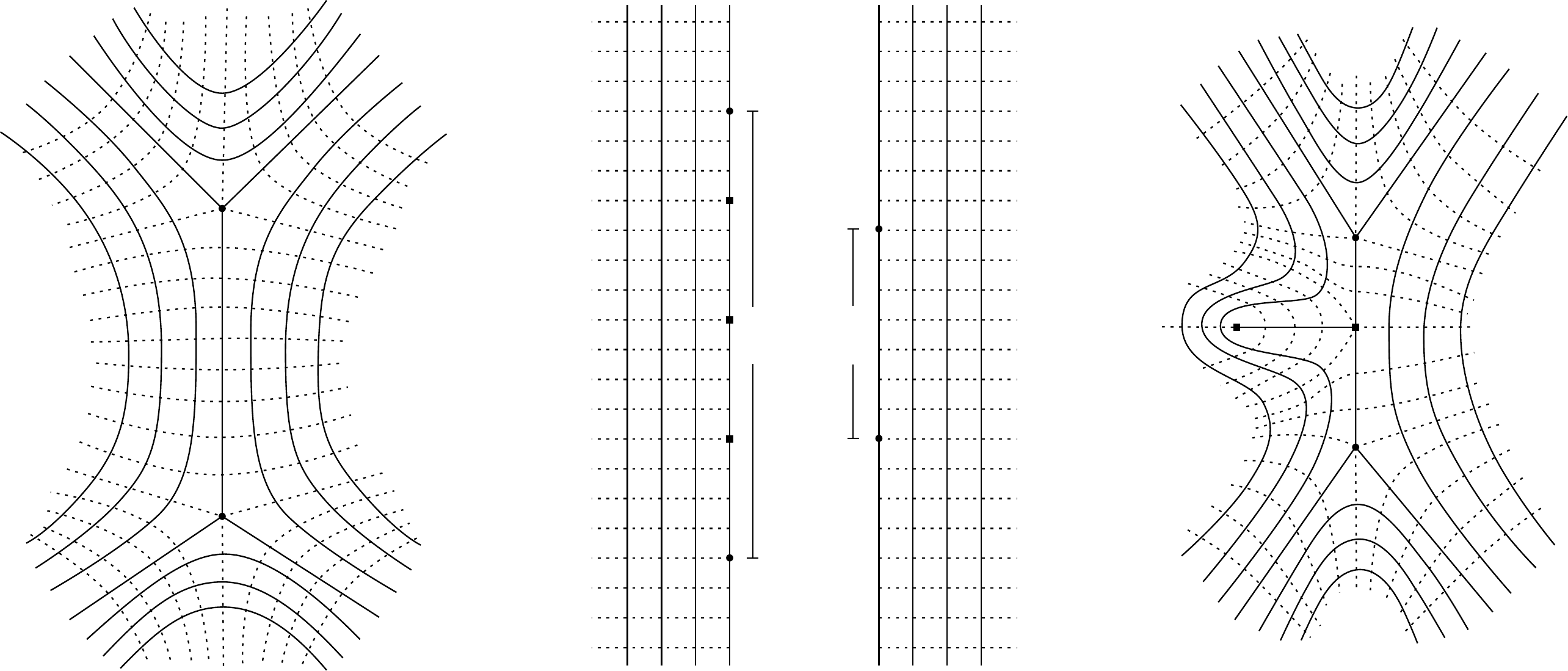}}%
    \put(0.47209055,0.20849207){\color[rgb]{0,0,0}\makebox(0,0)[lt]{\lineheight{0}\smash{\begin{tabular}[t]{l}$\ell_i$\end{tabular}}}}%
    \put(0.53651329,0.20771593){\color[rgb]{0,0,0}\makebox(0,0)[lt]{\lineheight{0}\smash{\begin{tabular}[t]{l}$\ell_j$\end{tabular}}}}%
  \end{picture}%
\endgroup%

        \caption{After rescaling the leaves, a saddle connection may have two different lengths on the boundaries of two adjacent components of $V(q)$. These components are glued together by introducing an additional saddle connection (right).}
        \label{fig:saddle_connection}
    \end{figure}

    After rescaling each leaf of $V(q)$ by the corresponding factor, we glue together all cylinders and minimal domains of $V(q)$ along their boundaries, introducing saddle connections when necessary. In this way, we obtain a measured foliation $V'$ from $V(q)$. It is clear that $V'$ is Whitehead equivalent to $V(q)$, since $V'$ is obtained from $V(q)$ by a sequence of Whitehead moves. In addition, we construct a measured foliation $F'$ by gluing the leaves of $H(q)$, restricted to the cylinders and minimal domains of $V(q)$, along their boundaries. The transverse measure of $F'$ is determined by the lengths of arcs on the leaves of $V'$. Consequently, we obtain a measured foliation $F'$ from $H(q)$ satisfying that 
    $$
    a_ji(G_j,F')=c_ji(G_j,H(q))
    $$
    for each $j=1,\cdots,n$. Then, 
    $$
    \sup_{F'\in\mathcal{MF}}\frac{i(F,F')^2}{\sum_{j=1}^n\frac{a_ji(G_j,F')^2}{i(G_j,H(q))}}=\sum_{j=1}^n\frac{c_j^2i(G_j,H(q))}{a_j}.
    $$

    By Theorem \ref{limit of e^2text}, we have
    $$
    \lim_{t\to\infty}e^{2t}\mathrm{Ext}_{X_t}(F)=\sum_{j=1}^n\frac{c_j^2i(G_j,H(q))}{a_j}
    $$
    for any $F=\sum_{j=1}^nc_jG_j$ in $\mathcal{MF}$, where $c_j> 0$.
\end{proof}


\section{The asymptotic behavior of pairs of \teich rays}\label{section4}

The divergence case of pairs of \teich rays is clear by the works of Ivanov \cite{Iva2001}, Lenzhen and Masur \cite{LM2010}. Based on Kerckhoff's formula, we study the limit of \teich distance between two points moving along two bounded \teich rays, respectively.

Given a measured foliation $F_0\in\mathcal{MF}$, the space of measured foliation space $\mathcal{MF}$ can be partitioned into two complementary subspaces based on the intersection properties of measured foliations in $\mathcal{MF}$ with respect to $F_0$. Then, we have
$$
\mathcal{MF}=\mathcal{MF}_0(F_0)\cup\mathcal{MF}_1(F_0),
$$
where
$$
\mathcal{MF}_0(F_0)=\{F\in\mathcal{MF}\mid i(F_0,F)=0\},
$$
and 
$$
\mathcal{MF}_1(F_0)=\{F\in\mathcal{MF}\mid i(F_0,F)\neq 0\}.
$$
If two measured foliations $F,F'\in\mathcal{MF}$ are absolutely continuous, it is evident that 
$$
\mathcal{MF}_0(F)=\mathcal{MF}_0(F') \text{ and } \mathcal{MF}_1(F)=\mathcal{MF}_1(F').
$$

Let $\mathcal{R}_{q,X}(t)$ be a \teich ray determined by a unit-norm holomorphic quadratic differential $q$ on $X$. The vertical foliation of $q$ is $V(q)=\sum_{j=1}^na_jG_j$, where $a_j\geq 0$. We define a function $\mathcal{E}_q:\mathcal{MF}\to\mathbb{R}_{\geq 0}$ associated with the \teich ray, which is given by
$$
\mathcal{E}_q(F)=\left\{\sum_{j=1}^n\frac{a_ji(G_j,F)^2}{i(G_j,H(q))}\right\}^{\frac{1}{2}}
$$
for any $F\in\mathcal{MF}$.

\begin{lemma}\label{inequality of Ext in MF_1}
    Let $\mathcal{R}_{q,X}(t)$ and $\mathcal{R}_{q',Y}(t)$ be two \teich rays induced by unit-norm quadratic differentials $q$ and $q'$, respectively. If the vertical foliations $V(q)=\sum_{j=1}^na_jG_j$ and $V(q')=\sum_{j=1}^nb_jG_j$ are absolutely continuous, where $a_j,b_j>0$, then
    $$
    \lim_{t\to\infty}\frac{\mathrm{Ext}_{Y_t}(F)}{\mathrm{Ext}_{X_t}(F)}=\frac{\mathcal{E}_{q'}(F)^2}{\mathcal{E}_q(F)^2}\leq\max_{1\leq j\leq n}\frac{b_ji(G_j,H(q))}{a_ji(G_j,H(q'))}
    $$
    for any $F\in\mathcal{MF}_1(V(q))$.
\end{lemma}

\begin{proof}
    By Theorem \ref{Walsh's equality}, for any $F\in\mathcal{MF}_1(V(q))$, there is 
    $$
    \lim_{t\to\infty}\frac{\mathrm{Ext}_{Y_t}(F)}{\mathrm{Ext}_{X_t}(F)}=\lim_{t\to\infty}\frac{e^{-2t}\mathrm{Ext}_{Y_t}(F)}{e^{-2t}\mathrm{Ext}_{X_t}(F)}=\frac{\sum_{j=1}^n\frac{b_ji(G_j,F)^2}{i(G_j,H(q'))}}{\sum_{j=1}^n\frac{a_ji(G_j,F)^2}{i(G_j,H(q))}}=\frac{\mathcal{E}_{q'}(F)^2}{\mathcal{E}_q(F)^2}.
    $$
    By relabeling the indices, we suppose that for all $j=1,\cdots,n,$
    $$
    \frac{\frac{b_j}{i(G_j,H(q'))}}{\frac{a_j}{i(G_j,H(q))}}\leq\frac{\frac{b_1}{i(G_1,H(q'))}}{\frac{a_1}{i(G_1,H(q))}}.
    $$
    Therefore,
    $$
    \frac{a_1}{i(G_1,H(q))}\sum_{j=1}^n\frac{b_ji(G_j,F)^2}{i(G_j,H(q'))}\leq\frac{b_1}{i(G_1,H(q'))}\sum_{j=1}^n\frac{a_ji(G_j,F)^2}{i(G_j,H(q))}.
    $$
    Then,
    $$
    \frac{\mathcal{E}_{q'}(F)^2}{\mathcal{E}_q(F)^2}=\frac{\sum_{j=1}^n\frac{b_ji(G_j,F)^2}{i(G_j,H(q'))}}{\sum_{j=1}^n\frac{a_ji(G_j,F)^2}{i(G_j,H(q))}}\leq\frac{\frac{b_1}{i(G_1,H(q'))}}{\frac{a_1}{i(G_1,H(q))}}=\max_{1\leq j\leq n}\frac{b_ji(G_j,H(q))}{a_ji(G_j,H(q'))}.
    $$
\end{proof}

Indeed, the inequality in Lemma \ref{inequality of Ext in MF_1} has been shown in \cite{Wal2019}. Since this inequality is crucial for our discussion, we state it and provide a proof here. Furthermore, Walsh established the following equality.

\begin{lemma}[{\cite[Lemma~29]{Wal2019}}]\label{maximum ratio of E_q}
    Let $q$ and $q'$ be two unit-norm holomorphic quadratic differentials on $X,Y\in\mathcal{T}(S)$, respectively. The vertical foliation of $q$ is $V(q)=\sum_{j=1}^na_jG_j$, where $a_j> 0$. If the vertical foliation of $q'$ can be expressed as $V(q')=\sum_{j=1}^nb_jG_j$, where $b_j\geq 0$, then 
    $$
    \sup_{F\in\mathcal{MF}_1(V(q))}\frac{\mathcal{E}_{q'}(F)^2}{\mathcal{E}_q(F)^2}=\max_{1\leq j\leq n}\frac{b_ji(G_j,H(q))}{a_ji(G_j,H(q'))}.
    $$
    Otherwise, the supremum is $+\infty$.
\end{lemma}

\begin{theorem}\label{equality of Ext in MF_1}
    Let $\mathcal{R}_{q,X}(t)$ and $\mathcal{R}_{q',Y}(t)$ be two \teich rays induced by unit-norm quadratic differentials $q$ and $q'$, respectively. If the vertical foliations $V(q)=\sum_{j=1}^na_jG_j$ and $V(q')=\sum_{j=1}^nb_jG_j$ are absolutely continuous, where $a_j,b_j>0$, then
    $$
    \lim_{t\to\infty}\sup_{F\in\mathcal{MF}_1(V(q))}\frac{\mathrm{Ext}_{Y_t}(F)}{\mathrm{Ext}_{X_t}(F)}=\max_{1\leq j\leq n}\frac{b_ji(G_j,H(q))}{a_ji(G_j,H(q'))}.
    $$
\end{theorem}

\begin{proof}
    By Theorem \ref{Walsh's equality} and Lemma \ref{maximum ratio of E_q}, we have
    \begin{equation}\label{limtinf of Ext in MF_1}
        \begin{aligned}
            \liminf_{t\to\infty}\sup_{F\in\mathcal{MF}_1(V(q))}\frac{\mathrm{Ext}_{Y_t}(F)}{\mathrm{Ext}_{X_t}(F)} 
            &\geq\sup_{F\in\mathcal{MF}_1(V(q))}\liminf_{t\to\infty}\frac{e^{-2t}\mathrm{Ext}_{Y_t}(F)}{e^{-2t}\mathrm{Ext}_{X_t}(F)} \\
            &=\sup_{F\in\mathcal{MF}_1(V(q))}\frac{\mathcal{E}_{q'}(F)^2}{\mathcal{E}_q(F)^2} \\
            &=\max_{1\leq j\leq n}\frac{b_ji(G_j,H(q))}{a_ji(G_j,H(q'))}.
        \end{aligned}
    \end{equation}

    By contradiction, we assume that
    $$
    \limsup_{t\to\infty}\sup_{F\in\mathcal{MF}_1(V(q))}\frac{\mathrm{Ext}_{Y_t}(F)}{\mathrm{Ext}_{X_t}(F)}>\max_{1\leq j\leq n}\frac{b_ji(G_j,H(q))}{a_ji(G_j,H(q'))}.
    $$
    For simplicity, let
    $$
    E=\limsup_{t\to\infty}\sup_{F\in\mathcal{MF}_1(V(q))}\frac{\mathrm{Ext}_{Y_t}(F)}{\mathrm{Ext}_{X_t}(F)} \text{ and } M=\max_{1\leq j\leq n}\frac{b_ji(G_j,H(q))}{a_ji(G_j,H(q'))}.
    $$
    There exists a subsequence $\{t_k\}_{k=1}^\infty$ such that 
    $$
    \lim_{k\to\infty}\sup_{F\in\mathcal{MF}_1(V(q))}\frac{\mathrm{Ext}_{Y_{t_k}}(F)}{\mathrm{Ext}_{X_{t_k}}(F)}=E>M+\frac{E-M}{3}.
    $$
    Then, there is a $N_0$ such that for any $k>N_0$, we have 
    $$
    \sup_{F\in\mathcal{MF}_1(V(q))}\frac{\mathrm{Ext}_{Y_{t_k}}(F)}{\mathrm{Ext}_{X_{t_k}}(F)}>M+\frac{E-M}{3}.
    $$
    For any $k>N_0$, there exists a $F_k\in\mathcal{MF}_1(V(q))$ such that
    \begin{equation}\label{contradiction 1 for Ext in MF_1}
        \frac{\mathrm{Ext}_{Y_{t_k}}(F_k)}{\mathrm{Ext}_{X_{t_k}}(F_k)}>\sup_{F\in\mathcal{MF}_1(V(q))}\frac{\mathrm{Ext}_{Y_{t_k}}(F)}{\mathrm{Ext}_{X_{t_k}}(F)}-\frac{E-M}{6}>M+\frac{E-M}{6}.
    \end{equation}
    By Lemma \ref{inequality of Ext in MF_1}, for any $F\in\mathcal{MF}_1(V(q))$, we have
    $$
    \lim_{k\to\infty}\frac{\mathrm{Ext}_{Y_{t_k}}(F)}{\mathrm{Ext}_{X_{t_k}}(F)}\leq M.
    $$
    Then, for $\varepsilon=\frac{E-M}{6}$, there is a $N_1>N_0>0$ such that for any $k>N_1$, 
    $$
    \frac{\mathrm{Ext}_{Y_{t_k}}(F)}{\mathrm{Ext}_{X_{t_k}}(F)}<M+\varepsilon=M+\frac{E-M}{6}.
    $$
    For each $k>N_1$, we choose the same $F_k\in\mathcal{MF}_1(V(q))$ as in \eqref{contradiction 1 for Ext in MF_1}. Therefore, we have
    \begin{equation}\label{contradiction 2 for Ext in MF_1}
        \frac{\mathrm{Ext}_{Y_{t_k}}(F_k)}{\mathrm{Ext}_{X_{t_k}}(F_k)}<M+\frac{E-M}{6}
    \end{equation} 
    for any $k>N_1$. We derive a contradiction between \eqref{contradiction 1 for Ext in MF_1} and \eqref{contradiction 2 for Ext in MF_1}. Then, there is 
    \begin{equation}\label{limtsup of Ext in MF_1}
        \limsup_{t\to\infty}\sup_{F\in\mathcal{MF}_1(V(q))}\frac{\mathrm{Ext}_{Y_t}(F)}{\mathrm{Ext}_{X_t}(F)}\leq\max_{1\leq j\leq n}\frac{b_ji(G_j,H(q))}{a_ji(G_j,H(q'))}.
    \end{equation}
    By combining \eqref{limtinf of Ext in MF_1} and \eqref{limtsup of Ext in MF_1}, we can obtain that 
    $$
    \lim_{t\to\infty}\sup_{F\in\mathcal{MF}_1(V(q))}\frac{\mathrm{Ext}_{Y_t}(F)}{\mathrm{Ext}_{X_t}(F)}=\max_{1\leq j\leq n}\frac{b_ji(G_j,H(q))}{a_ji(G_j,H(q'))}.
    $$
\end{proof}

Subsequently, we consider the ratio of extremal lengths along a pair of \teich rays $\mathcal{R}_{q,X}(t)$ and $\mathcal{R}_{q',Y}(t)$, for measured foliations in $\mathcal{MF}_0(V(q))$.

\begin{lemma}\label{sup in MF'_0}
    Let $q$ and $q'$ be two unit-norm holomorphic quadratic differentials on $X, Y\in\mathcal{T}(S)$, respectively. If the vertical foliations $V(q)=\sum_{j=1}^na_jG_j$ and $V(q')=\sum_{j=1}^nb_jG_j$ are absolutely continuous, where $a_j,b_j>0$, then
    $$
    \sup_{F\in\mathcal{MF}'_0(V(q))}\frac{\sum_{j=1}^n\frac{c_j^2i(G_j,H(q'))}{b_j}}{\sum_{j=1}^n\frac{c_j^2i(G_j,H(q))}{a_j}}=\max_{1\leq j\leq n}\frac{a_ji(G_j,H(q'))}{b_ji(G_j,H(q))},
    $$
    where $\mathcal{MF}'_0(V(q))=\{F\in\mathcal{MF}\mid F=\sum_{j=1}^nc_jG_j,\;c_j> 0\}$.
\end{lemma}

\begin{proof}
    By relabeling the indices, we can suppose that for any $j=1,\cdots,n$,
    $$
    \frac{\frac{i(G_j,H(q'))}{b_j}}{\frac{i(G_j,H(q))}{a_j}}\leq\frac{\frac{i(G_1,H(q'))}{b_1}}{\frac{i(G_1,H(q))}{a_1}}.
    $$
    Then, for any $F\in\mathcal{MF}'_0(V(q))$, there is 
    $$
    \frac{i(G_1,H(q))}{a_1}\sum_{j=1}^n\frac{c_j^2i(G_j,H(q'))}{b_j}\leq\frac{i(G_1,H(q'))}{b_1}\sum_{j=1}^n\frac{c_j^2i(G_j,H(q))}{a_j}.
    $$
    Therefore, 
    $$
    \frac{\sum_{j=1}^n\frac{c_j^2i(G_j,H(q'))}{b_j}}{\sum_{j=1}^n\frac{c_j^2i(G_j,H(q))}{a_j}}\leq\frac{\frac{i(G_1,H(q'))}{b_1}}{\frac{i(G_1,H(q))}{a_1}}=\max_{1\leq j\leq n}\frac{a_ji(G_j,H(q'))}{b_ji(G_j,H(q))}.
    $$
    Moreover, we can choose a measured foliation $F\in\mathcal{MF}'_0(V(q))$ such that $c_1>0$ and $c_j=\varepsilon$ for all $j=2,\cdots,n$. Then, there is  
    $$
    \frac{\frac{c_1^2i(G_1,H(q'))}{b_1}+\varepsilon^2\sum_{j=2}^n\frac{i(G_j,H(q'))}{b_j}}{\frac{c_1^2i(G_1,H(q))}{a_1}+\varepsilon^2\sum_{j=2}^n\frac{i(G_j,H(q))}{a_j}}\leq\frac{\frac{i(G_1,H(q'))}{b_1}}{\frac{i(G_1,H(q))}{a_1}}=\max_{1\leq j\leq n}\frac{a_ji(G_j,H(q'))}{b_ji(G_j,H(q))}.
    $$
    Since the $\varepsilon$ can be chosen arbitrarily small, we can conclude that
    $$
    \sup_{F\in\mathcal{MF}'_0(V(q))}\frac{\sum_{j=1}^n\frac{c_j^2i(G_j,H(q'))}{b_j}}{\sum_{j=1}^n\frac{c_j^2i(G_j,H(q))}{a_j}}=\max_{1\leq j\leq n}\frac{a_ji(G_j,H(q'))}{b_ji(G_j,H(q))}.
    $$
\end{proof}

\begin{theorem}\label{equality of Ext in MF_0}
    Let $\mathcal{R}_{q,X}(t)$ and $\mathcal{R}_{q',Y}(t)$ be two \teich rays induced by unit-norm quadratic differentials $q$ and $q'$, respectively. If the vertical foliations $V(q)=\sum_{j=1}^na_jG_j$ and $V(q')=\sum_{j=1}^nb_jG_j$ are absolutely continuous, where $a_j,b_j>0$, then
    $$
    \lim_{t\to\infty}\sup_{F\in\mathcal{MF}_0(V(q))}\frac{\mathrm{Ext}_{Y_t}(F)}{\mathrm{Ext}_{X_t}(F)}=\max_{1\leq j\leq n}\frac{a_ji(G_j,H(q'))}{b_ji(G_j,H(q))}.
    $$
\end{theorem}

\begin{proof}
    By Lemma \ref{inequality}, for any $F,F'\in\mathcal{MF}$, there are 
    $$
    \mathrm{Ext}_{X_t}(F)\mathrm{Ext}_{X_t}(F')\geq i(F,F')^2.
    $$
    Then, we have
    $$
    e^{2t}\mathrm{Ext}_{X_t}(F)\geq\frac{i(F,F')^2}{e^{-2t}\mathrm{Ext}_{X_t}(F')}.
    $$
    Moreover, for any $F\in\mathcal{MF}$, there exists a $H'_t\in\mathcal{MF}$ such that 
    $$
    \mathrm{Ext}_{Y_t}(F)\mathrm{Ext}_{Y_t}(H'_t)=i(F,H'_t)^2.
    $$
    Then, there is
    $$
    e^{2t}\mathrm{Ext}_{Y_t}(F)=\frac{i(F,H'_t)^2}{e^{-2t}\mathrm{Ext}_{Y_t}(H'_t)}.
    $$

    Therefore, we have
    \begin{equation}\label{limtsup-1 of Ext in MF_0}
        \begin{aligned}
            \limsup_{t\to\infty}\sup_{F\in\mathcal{MF}_0(V(q))}\frac{\mathrm{Ext}_{Y_t}(F)}{\mathrm{Ext}_{X_t}(F)}
            &=\limsup_{t\to\infty}\sup_{F\in\mathcal{MF}_0(V(q))}\frac{e^{2t}\mathrm{Ext}_{Y_t}(F)}{e^{2t}\mathrm{Ext}_{X_t}(F)} \\
            &=\limsup_{t\to\infty}\sup_{F\in\mathcal{MF}_0(V(q))}\frac{\frac{i(F,H'_t)^2}{e^{-2t}\mathrm{Ext}_{Y_t}(H'_t)}}{e^{2t}\mathrm{Ext}_{X_t}(F)} \\
            &\leq\limsup_{t\to\infty}\sup_{F\in\mathcal{MF}_0(V(q))}\frac{\frac{i(F,H'_t)^2}{e^{-2t}\mathrm{Ext}_{Y_t}(H'_t)}}{\frac{i(F,H'_t)^2}{e^{-2t}\mathrm{Ext}_{X_t}(H'_t)}} \\
            &=\limsup_{t\to\infty}\frac{e^{-2t}\mathrm{Ext}_{X_t}(H'_t)}{e^{-2t}\mathrm{Ext}_{Y_t}(H'_t)} \\
            &\leq\limsup_{t\to\infty}\sup_{F\in\mathcal{MF}_1(V(q))}\frac{\mathrm{Ext}_{X_t}(F)}{\mathrm{Ext}_{Y_t}(F)}.
        \end{aligned}
    \end{equation}

    By contradiction, we assume that
    $$
    \limsup_{t\to\infty}\sup_{F\in\mathcal{MF}_1(V(q))}\frac{\mathrm{Ext}_{X_t}(F)}{\mathrm{Ext}_{Y_t}(F)}>\max_{1\leq j\leq n}\frac{a_ji(G_j,H(q'))}{b_ji(G_j,H(q))}.
    $$
    For simplicity, let
    $$
    E'=\limsup_{t\to\infty}\sup_{F\in\mathcal{MF}_1(V(q))}\frac{\mathrm{Ext}_{X_t}(F)}{\mathrm{Ext}_{Y_t}(F)} \text{ and } M'=\max_{1\leq j\leq n}\frac{a_ji(G_j,H(q'))}{b_ji(G_j,H(q))}.
    $$
    There exists a subsequence $\{t_k\}_{k=1}^\infty$ such that
    $$
    \lim_{k\to\infty}\sup_{F\in\mathcal{MF}_1(V(q))}\frac{\mathrm{Ext}_{X_{t_k}}(F)}{\mathrm{Ext}_{Y_{t_k}}(F)}=E'>M'+\frac{E'-M'}{3}.
    $$
    Then, there is a $N'_0$ such that for any $k>N'_0$, we have 
    $$
    \sup_{F\in\mathcal{MF}_1(V(q))}\frac{\mathrm{Ext}_{X_{t_k}}(F)}{\mathrm{Ext}_{Y_{t_k}}(F)}>M'+\frac{E'-M'}{3}.
    $$
    For any $k>N'_0$, there exists a $F_k\in\mathcal{MF}_1(V(q))$ such that
    \begin{equation}\label{contradiction 1 for Ext in MF_0}
        \frac{\mathrm{Ext}_{X_{t_k}}(F_k)}{\mathrm{Ext}_{Y_{t_k}}(F_k)}>\sup_{F\in\mathcal{MF}_1(V(q))}\frac{\mathrm{Ext}_{X_{t_k}}(F)}{\mathrm{Ext}_{Y_{t_k}}(F)}-\frac{E'-M'}{6}>M'+\frac{E'-M'}{6}.
    \end{equation}
    By Lemma \ref{inequality of Ext in MF_1}, for any $F\in\mathcal{MF}_1(V(q))$, there is 
    $$
    \lim_{t\to\infty}\frac{\mathrm{Ext}_{X_t}(F)}{\mathrm{Ext}_{Y_t}(F)}\leq M'.
    $$
    Then, for the subsequence $\{t_k\}_{k=1}^\infty$, we still have
    $$
    \lim_{k\to\infty}\frac{\mathrm{Ext}_{X_{t_k}}(F)}{\mathrm{Ext}_{Y_{t_k}}(F)}\leq M'.
    $$
    Then, for $\varepsilon=\frac{E'-M'}{6}$, there is a $N'_1>N'_0>0$ such that for any $k>N'_1$, 
    $$
    \frac{\mathrm{Ext}_{X_{t_k}}(F)}{\mathrm{Ext}_{Y_{t_k}}(F)}<M'+\varepsilon=M'+\frac{E'-M'}{6}.
    $$
    For each $k>N'_1$, we choose the same $F_k\in\mathcal{MF}_1(V(q))$ as in \eqref{contradiction 1 for Ext in MF_0}. Therefore, we have
    \begin{equation}\label{contradiction 2 for Ext in MF_0}
        \frac{\mathrm{Ext}_{X_{t_k}}(F_k)}{\mathrm{Ext}_{Y_{t_k}}(F_k)}<M'+\frac{E'-M'}{6}
    \end{equation} 
    for any $k>N'_1$. We derive a contradiction between \eqref{contradiction 1 for Ext in MF_0} and \eqref{contradiction 2 for Ext in MF_0}. Then, there is 
    \begin{equation}\label{limtsup-2 of Ext in MF_0}
        \limsup_{t\to\infty}\sup_{F\in\mathcal{MF}_1(V(q))}\frac{\mathrm{Ext}_{X_t}(F)}{\mathrm{Ext}_{Y_t}(F)}\leq\max_{1\leq j\leq n}\frac{a_ji(G_j,H(q'))}{b_ji(G_j,H(q))}.
    \end{equation}
    By combining \eqref{limtsup-1 of Ext in MF_0} and \eqref{limtsup-2 of Ext in MF_0}, we can obtain that 
    \begin{equation}\label{limtsup of Ext in MF_0}
        \limsup_{t\to\infty}\sup_{F\in\mathcal{MF}_0(V(q))}\frac{\mathrm{Ext}_{Y_t}(F)}{\mathrm{Ext}_{X_t}(F)}\leq\max_{1\leq j\leq n}\frac{a_ji(G_j,H(q'))}{b_ji(G_j,H(q))}.
    \end{equation}

    Moreover, by Corollary \ref{limit for similar foliation} and Lemma \ref{sup in MF'_0}, we have
    \begin{equation}\label{limtinf of Ext in MF_0}
        \begin{aligned}
            \liminf_{t\to\infty}\sup_{F\in\mathcal{MF}_0(V(q))}\frac{\mathrm{Ext}_{Y_t}(F)}{\mathrm{Ext}_{X_t}(F)}
            &\geq\sup_{F\in\mathcal{MF}_0(V(q))}\liminf_{t\to\infty}\frac{\mathrm{Ext}_{Y_t}(F)}{\mathrm{Ext}_{X_t}(F)} \\
            &\geq\sup_{F\in\mathcal{MF}'_0(V(q))}\liminf_{t\to\infty}\frac{e^{2t}\mathrm{Ext}_{Y_t}(F)}{e^{2t}\mathrm{Ext}_{X_t}(F)} \\
            &=\sup_{F\in\mathcal{MF}'_0(V(q))}\frac{\sum_{j=1}^n\frac{c_j^2i(G_j,H(q'))}{b_j}}{\sum_{j=1}^n\frac{c_j^2i(G_j,H(q))}{a_j}} \\
            &=\max_{1\leq j\leq n}\frac{a_ji(G_j,H(q'))}{b_ji(G_j,H(q))}.
        \end{aligned}
    \end{equation}
    By combining \eqref{limtsup of Ext in MF_0} and \eqref{limtinf of Ext in MF_0}, we conclude that 
    $$
    \lim_{t\to\infty}\sup_{F\in\mathcal{MF}_0(V(q))}\frac{\mathrm{Ext}_{Y_t}(F)}{\mathrm{Ext}_{X_t}(F)}=\max_{1\leq j\leq n}\frac{a_ji(G_j,H(q'))}{b_ji(G_j,H(q))}.
    $$
\end{proof}

\begin{proof}[\textbf{Proof of Theorem} \ref{limiting distance}]
    If the vertical foliations $V(q)$ and $V(q')$ are absolutely continuous, then by Theorem \ref{equality of Ext in MF_1}, Theorem \ref{equality of Ext in MF_0} and Kerckhoff's formula, we have
    \begin{align*}
        \lim_{t\to\infty}d_{\mathcal{T}}(X_t,Y_t)
        &=\lim_{t\to\infty}\frac{1}{2}\log\sup_{F\in\mathcal{MF}}\frac{\mathrm{Ext}_{Y_t}(F)}{\mathrm{Ext}_{X_t}(F)} \\
        &=\frac{1}{2}\log\lim_{t\to\infty}\max\left\{\sup_{F\in\mathcal{MF}_0(V(q))}\frac{\mathrm{Ext}_{Y_t}(F)}{\mathrm{Ext}_{X_t}(F)},\sup_{F\in\mathcal{MF}_1(V(q))}\frac{\mathrm{Ext}_{Y_t}(F)}{\mathrm{Ext}_{X_t}(F)}\right\} \\
        &=\frac{1}{2}\log\max\left\{\lim_{t\to\infty}\sup_{F\in\mathcal{MF}_0(V(q))}\frac{\mathrm{Ext}_{Y_t}(F)}{\mathrm{Ext}_{X_t}(F)},\lim_{t\to\infty}\sup_{F\in\mathcal{MF}_1(V(q))}\frac{\mathrm{Ext}_{Y_t}(F)}{\mathrm{Ext}_{X_t}(F)}\right\} \\
        &=\frac{1}{2}\log\max\left\{\max_{1\leq j\leq n}\frac{a_ji(G_j,H(q'))}{b_ji(G_j,H(q))},\max_{1\leq j\leq n}\frac{b_ji(G_j,H(q))}{a_ji(G_j,H(q'))}\right\} \\
        &=\frac{1}{2}\log\max_{1\leq j\leq n}\left\{\frac{a_ji(G_j,H(q'))}{b_ji(G_j,H(q))},\frac{b_ji(G_j,H(q))}{a_ji(G_j,H(q'))}\right\}.
    \end{align*}
    Otherwise, by the results of Ivanov \cite{Iva2001}, Lenzhen and Masur \cite{LM2010}, the \teich distance $d_{\mathcal{T}}(X_t,Y_t)$ tends to infinity as $t\to\infty$.
\end{proof}

\begin{proof}[\textbf{Proof of Corollary} \ref{asymptotic}]
    If the \teich rays $\mathcal{R}_{q,X}(t)$ and $\mathcal{R}_{q',Y}(t)$ are asymptotic, then without loss of generality, we assume that their initial points are appropriately selected such that 
    $$
    \lim_{t\to\infty}d_{\mathcal{T}}(X_t,Y_t)=0.
    $$
    Therefore, by Theorem \ref{limiting distance}, we have 
    $$
    \frac{a_j}{i(G_j,H(q))}=\frac{b_j}{i(G_j,H(q'))},
    $$
    for all $j=1,\cdots,n$. Then $q$ and $q'$ are modularly equivalent.

    Conversely, if $q$ and $q'$ are modularly equivalent, then there exists a constant $C>0$ such that 
    $$
    \frac{a_j}{i(G_j,H(q))}=C\frac{b_j}{i(G_j,H(q'))},
    $$
    for all $j=1,\cdots,n$. Let $\sigma=\frac{1}{2}\log C$. There is 
    $$
    \lim_{t\to\infty}d_{\mathcal{T}}(X_t,Y_{t+\sigma})=\frac{1}{2}\log\max_{1\leq j\leq n}\left\{\frac{a_ji(G_j,H(q'))}{e^{2\sigma}b_ji(G_j,H(q))},\frac{e^{2\sigma}b_ji(G_j,H(q))}{a_ji(G_j,H(q'))}\right\}=0.
    $$
    Then the \teich rays $\mathcal{R}_{q,X}(t)$ and $\mathcal{R}_{q',Y}(t)$ are asymptotic.
\end{proof}


\section{The minimum of the limiting \teich distance}\label{section5}

Since the limiting \teich distance between two points along a pair of \teich rays depends on the initial points of these rays, the minimum of the limiting \teich distance can be attained by shifting the initial points along the corresponding \teich geodesics. This minimum value is related to the detour metric $\delta$ between the two endpoints of the \teich rays on the horofunction boundary of $\mathcal{T}(S)$.

\subsection{The horofunction compactification}

We recall the horofunction compactification of a metric space which is introduced by Gromov in \cite{Gro1981}. Consider a proper geodesic metric space $(M,d)$ and choose a basepoint $b\in M$. A metric space is proper if all of its closed balls are compact, and it is geodesic if any two points can be connected by a geodesic segment. Define a function $\psi_z:M\to\mathbb{R}$ which is given by
$$
\psi_z(x):=d(x,z)-d(b,z),\;\forall x\in M.
$$
Then, this defines a mapping
\begin{equation*}
    \begin{array}{cccc}
        \Psi: & M &\to &C(M) \\
         &z &\mapsto &\psi_z,
    \end{array}
\end{equation*}
where $C(M)$ denotes the space of all continuous functions on $M$, equipped with the topology of uniform convergence on any compact subset of M.

Gromov showed that the mapping $\Psi$ is an embedding and the set $\Psi(M)$ is relatively compact in $C(M)$. The closure $\overline{\Psi(M)}$ is called the horofunction compactification of $M$. The horofunction boundary of $M$ is defined as $\partial_{hor}M:=\overline{\Psi(M)}\setminus\Psi(M)$, and each $\xi\in\partial_{hor}M$ is called a horofunction.

The horofunction boundary of $M$ depends on the choice of basepoint in $M$, and the horofunction boundaries of $M$ from different basepoints are homeomorphic. This homeomorphism is realized by applying additive constant adjustment to each horofunction.

Since the \teich space $\mathcal{T}(S)$ with \teich metric $d_{\mathcal{T}}$ is known to be a proper geodesic metric space, we can consider its horofunction compactification, denoted by $\overline{\mathcal{T}(S)}^{hor}$. The horofunction boundary of $\mathcal{T}(S)$ is denoted by $\partial_{hor}\mathcal{T}(S)$.

The Gardiner-Masur compactification of \teich space is introduced by Gardiner and Masur in \cite{GM1991}. By using extremal lengths of simple closed curves, we can define a mapping
\begin{equation*}
    \begin{array}{cccl}
        \varphi: & \mathcal{T}(S) & \to & \mathbb{R}_{\geq 0}^{\mathcal{S}}  \\
        & X & \mapsto & \left\{\mathrm{Ext}_X(\alpha)^{\frac{1}{2}}\right\}_{\alpha\in\mathcal{S}}. 
    \end{array}
\end{equation*}
Gardiner and Masur proved that the mapping $\Phi:=\pi\circ\phi:\mathcal{T}(S)\to P\mathbb{R}_{\geq 0}^\mathcal{S}$ is an embedding, where $\pi:\mathbb{R}_{\geq 0}^\mathcal{S}\setminus\{0\}\to P\mathbb{R}_{\geq 0}^\mathcal{S}$ is the natural projection, and the closure $\overline{\Phi(\mathcal{T}(S))}$ is compact. The Gardiner-Masur compactification of $\mathcal{T}(S)$, denoted by $\overline{\mathcal{T}(S)}^{GM}$, is defined by the closure $\overline{\Phi(\mathcal{T}(S))}$, and the Gardiner-Masur boundary of $\mathcal{T}(S)$ is defined as $\partial_{GM}\mathcal{T}(S):=\overline{\Phi(\mathcal{T}(S))}\setminus\Phi(\mathcal{T}(S))$.

Liu and Su \cite{LS2014} showed that the horofunction compactification of \teich space is homeomorphic to the Gardiner-Masur compactification, which was also proved independently by Walsh in \cite{Wal2019}. Walsh also showed that the \teich ray $\mathcal{R}_{q,X}(t)$ converges to the projective class of $\mathcal{E}_q(\cdot)$ in $\overline{\mathcal{T}(S)}^{GM}$.

\subsection{The detour metric}

Let $\mathcal{I}$ be an unbounded subset of $\mathbb{R}_{\geq 0}$ containing $0$. An almost-geodesic ray in a metric space $(M,d)$ is defined as a mapping $\gamma:\mathcal{I}\to M$ satisfying that for any $\varepsilon>0$, there exists a $T\geq 0$ such that 
$$
|d(\gamma(0),\gamma(s))+d(\gamma(s),\gamma(t))-t|<\varepsilon,
$$
for all $s,t\in E$ with $t\geq s\geq T$.

Rieffel \cite{Rie2002} proved that each almost-geodesic ray in a metric space $(M,d)$ converges to a point in $\partial_{hor}M$. A horofunction in $\partial_{hor}M$ is called a Busemann point if it is the limit of an almost-geodesic ray. The set of all Busemann points in $\partial_{hor}M$, called the Busemann set of $M$, is denoted by $\partial_B M$.

For any two horofunctions $\xi, \eta\in\partial_{hor}M$, the detour cost is defined by 
$$
H(\xi,\eta):=\sup_{W\ni\xi}\inf_{x\in W}\left(d(b,x)+\eta(x)\right),
$$
where $W$ ranges over all neighborhoods of $\xi$ in the horofunction compactification of $(M,d)$. An equivalent definition is given by
$$
H(\xi,\eta):=\inf_{\gamma}\liminf_{t\to\infty}\left(d(b,\gamma(t))+\eta(\gamma(t))\right),
$$
where the infimum is taken over all paths $\gamma:\mathbb{R}_{\geq 0}\to M$ converging to $\xi$.

For any two Busemann points $\xi, \eta\in\partial_B M$, the detour distance between them is defined as 
$$
\delta(\xi,\eta)=H(\xi,\eta)+H(\eta,\xi).
$$
We refer to \cite{Wal2014} for more details. The detour metric $\delta$ may take infinity, and it is independent of the basepoint in $M$.

Walsh \cite{Wal2019} showed that for any $X\in\mathcal{T}(S)$ and $\xi\in\partial_B\mathcal{T}(S)$, there exists a unique \teich ray starting from $X$ and converging to $\xi$. We denote by $\mathcal{B}_q$ the Busemann point in $\partial_B\mathcal{T}(S)$, which is the limit of a \teich ray $\mathcal{R}_{q,X}(t)$. Walsh also established a formula of detour distance between two Busemann points which are the limits of two \teich rays from the basepoint of $\overline{\mathcal{T}(S)}^{hor}$ (see \cite[Corollary~3]{Wal2019}). Amano \cite{Ama2014} generalized this formula to any two Busemann points associated with distinct \teich rays originating from different satrting points.

\begin{proposition}[{\cite[Proposition~4.14]{Ama2014}}]\label{detour distance}
    Let $\mathcal{R}_{q,X}(t)$ and $\mathcal{R}_{q',Y}(t)$ be two \teich rays. If the vertical foliations $V(q)=\sum_{j=1}^na_jG_j$ and $V(q')=\sum_{j=1}^nb_jG_j$ are absolutely continuous, where $a_j, b_j>0$, then the detour distance between $\mathcal{B}_q$ and $\mathcal{B}_{q'}$ associated with these two rays is 
    $$
    \delta(\mathcal{B}_q,\mathcal{B}_{q'})=\frac{1}{2}\log\max_{1\leq j\leq n}\frac{a_ji(G_j,H(q'))}{b_ji(G_j,H(q))}+\frac{1}{2}\log\max_{1\leq j\leq n}\frac{b_ji(G_j,H(q))}{a_ji(G_j,H(q'))}.
    $$
    If $V(q)$ and $V(q')$ are not absolutely continuous, then $\delta(\mathcal{B}_q,\mathcal{B}_{q'})=+\infty$.
\end{proposition}

\begin{proof}[\textbf{Proof of Proposition} \ref{minimum of the limiting distance}]
    If $V(q)$ and $V(q')$ are absolutely continuous, then by Theorem \ref{limiting distance} and Proposition \ref{detour distance}, we have
    \begin{align*}
        \lim_{t\to\infty}d_{\mathcal{T}}(X_t,Y_t)&=\frac{1}{2}\log\max_{1\leq j\leq n}\left\{\frac{a_ji(G_j,H(q'))}{b_ji(G_j,H(q))},\frac{b_ji(G_j,H(q))}{a_ji(G_j,H(q'))}\right\} \\
        &\geq\frac{1}{2}\log\left(\left(\max_{1\leq j\leq n}\frac{a_ji(G_j,H(q'))}{b_ji(G_j,H(q))}\right)^{\frac{1}{2}}\cdot\left(\max_{1\leq j\leq n}\frac{b_ji(G_j,H(q))}{a_ji(G_j,H(q'))}\right)^{\frac{1}{2}}\right) \\
        &=\frac{1}{2}\left(\frac{1}{2}\log\max_{1\leq j\leq n}\frac{a_ji(G_j,H(q'))}{b_ji(G_j,H(q))}+\frac{1}{2}\log\max_{1\leq j\leq n}\frac{b_ji(G_j,H(q))}{a_ji(G_j,H(q'))}\right) \\
        &=\frac{1}{2}\delta(\mathcal{B}_q,\mathcal{B}_{q'}).
    \end{align*}
    Since the detour distance $\delta(\mathcal{B}_q,\mathcal{B}_{q'})$ is invariant when the initial points of the \teich rays $\mathcal{R}_{q,X}(t)$ and $\mathcal{R}_{q',Y}(t)$ are shifted along their corresponding \teich geodesics, then for any $\sigma\in\mathbb{R}$, there is 
    $$
    \lim_{t\to\infty}d_{\mathcal{T}}(X_t,Y_{t+\sigma})\geq\frac{1}{2}\delta(\mathcal{B}_q,\mathcal{B}_{q'}),
    $$
    where the equality holds if 
    $$
    \sigma=\frac{1}{4}\log\frac{\max_{1\leq j\leq n}\frac{a_ji(G_j,H(q'))}{b_ji(G_j,H(q))}}{\max_{1\leq j\leq n}\frac{b_ji(G_j,H(q))}{a_ji(G_j,H(q'))}}.
    $$
    Thus, for this $\sigma$, we can obtain that 
    \begin{align*}
        \max_{1\leq j\leq n}\frac{e^{2\sigma}b_ji(G_j,H(q))}{a_ji(G_j,H(q'))}&=\frac{\left(\max_{1\leq j\leq n}\frac{a_ji(G_j,H(q'))}{b_ji(G_j,H(q))}\right)^{\frac{1}{2}}}{\left(\max_{1\leq j\leq n}\frac{b_ji(G_j,H(q))}{a_ji(G_j,H(q'))}\right)^{\frac{1}{2}}}\max_{1\leq j\leq n}\frac{b_ji(G_j,H(q))}{a_ji(G_j,H(q'))} \\
        &=\left(\max_{1\leq j\leq n}\frac{a_ji(G_j,H(q'))}{b_ji(G_j,H(q))}\right)^{\frac{1}{2}}\cdot\left(\max_{1\leq j\leq n}\frac{b_ji(G_j,H(q))}{a_ji(G_j,H(q'))}\right)^{\frac{1}{2}} \\
        &=\max_{1\leq j\leq n}\frac{a_ji(G_j,H(q'))}{e^{2\sigma}b_ji(G_j,H(q))}.
    \end{align*}
    Therefore, we conclude that
    \begin{align*}
        \lim_{t\to\infty}d_{\mathcal{T}}(X_t,Y_{t+\sigma})&=\frac{1}{2}\log\max_{1\leq j\leq n}\left\{\frac{a_ji(G_j,H(q'))}{e^{2\sigma}b_ji(G_j,H(q))},\frac{e^{2\sigma}b_ji(G_j,H(q))}{a_ji(G_j,H(q'))}\right\} \\
        &=\frac{1}{2}\log\left(\left(\max_{1\leq j\leq n}\frac{a_ji(G_j,H(q'))}{b_ji(G_j,H(q))}\right)^{\frac{1}{2}}\cdot\left(\max_{1\leq j\leq n}\frac{b_ji(G_j,H(q))}{a_ji(G_j,H(q'))}\right)^{\frac{1}{2}}\right) \\
        &=\frac{1}{2}\left(\frac{1}{2}\log\max_{1\leq j\leq n}\frac{a_ji(G_j,H(q'))}{b_ji(G_j,H(q))}+\frac{1}{2}\log\max_{1\leq j\leq n}\frac{b_ji(G_j,H(q))}{a_ji(G_j,H(q'))}\right) \\
        &=\frac{1}{2}\delta(\mathcal{B}_q,\mathcal{B}_{q'}).
    \end{align*}
\end{proof}

\begin{proof}[\textbf{Proof of Corollary} \ref{identical Busemann points}]
    By Proposition \ref{minimum of the limiting distance}, the Busemann points $\mathcal{B}_q$ and $\mathcal{B}_{q'}$ are identical if and only if their associated \teich rays are asymptotic. Combining this with Corollary \ref{asymptotic}, we can obtain that the Busemann points $\mathcal{B}_q$ and $\mathcal{B}_{q'}$ are identical if and only if $q$ and $q'$ are modularly equivalent.
\end{proof}


\section*{Acknowledgments}
We are very grateful to Hideki Miyachi and Guangming Hu for their helpful discussions and invaluable encouragements.




\bibliography{Reference}
\bibliographystyle{plainnat-nourl}

\end{document}